\documentclass{article}

\renewcommand{\part}[2]{ \frac{\partial#1}{\partial#2} }

\newenvironment{proofof}[1]{
  \noindent\textbf{Proof of #1.}\ }{\hspace*{\fill}
  \begin{math}\Box\end{math}\medskip}

\usepackage{amsmath,amssymb,enumerate,bbm,calc,capt-of,ifthen}
\usepackage{amsmath} \usepackage{amsfonts} \usepackage{amsthm}
\usepackage{amssymb} \usepackage{mathrsfs} \newtheorem{thm}{Theorem}
\newtheorem{lemma}{Lemma} 
 \newtheorem{prop}{Proposition}
 \newcommand{\RR}{\mathbb{R}}
 \newcommand{\U}{\hat{\overline{u}}}
\newcommand{\Uf}{\overline{u}} \newcommand{\K}{\mathbf{k}}

\newcommand{\Rtpone}{{\mathbb{R}^{3+1}}} \newcommand{\kmax}{K_{max}}

\newcommand{\abs}[1]{\left| #1 \right|}

\title{Improved Error Bounds for Dirichlet-to-Neumann Absorbing
  Boundaries}

\author{Charles Siegel, Avy Soffer and Chris Stucchio}

\begin{document}

\maketitle
\begin{abstract}
  It has long been known how to construct radiation boundary
  conditions for the time dependent wave equation. Although arguments
  suggesting that they are accurate have been given, it is only
  recently that rigorous error bounds have been proved. Previous
  estimates show that the error caused by these methods behaves like
  $\epsilon C_\gamma e^{\gamma t}1$ for any $\gamma > 0$. We improve
  these results and show that the error behaves like $C \epsilon t^2$.
\end{abstract}

\section{Introduction}

Numerical solution of time dependent wave equations is an important
problem in physics, engineering and mathematics. To solve the wave
equation on $\Rtpone$, one must truncate the domain to a finite region
due to the limited memory of most computers. Of course, on a finite
region, boundary conditions must be specified in such a way as to minimize spurious reflections.
Boundary conditions of this form were first described in
\cite{MR596431,MR658635,MR0471386,MR517938,MR611807},  although rigorous error bounds would
wait until more recently \cite{MR1819643,MR2032866}.

In \cite{MR1819643}, a family of absorbing
boundary conditions based on rational function approximation to the
Dirichlet-to-Neumann operator in the frequency domain are reviewed. Boundary
conditions for the half-space (a boundary at $x=0$), as well as
cylindrical and spherical coordinates are also constructed.
Error bounds are proved for this family by inverting the
Fourier-Laplace transform for both the true solution and his
approximation and bounding the difference. Due to poles on the
imaginary line in $s$ ($s$ being the variable dual to $t$), the difference is bounded on a countour separated from the singular points, namely a line in the right half plane $\gamma + i \RR$. This shows that the error is bounded by $C_{\gamma} e^{\gamma t}$, with
$C_{\gamma}$ left implicit.

A careful examination of the poles of the rational function reveal that they approximate the
branch cut of the true solution in the sense of hyperfunctions
\cite{sterninshatalov:resurgentanalysis}. Instead of using the
machinery of hyperfunctions, we take an elementary approach. The true
solution can be represented as a certain integral over a compact
region. The approximate solution, after we collect the residues
associated to the poles on the imaginary line, turns out to be a
quadrature for this integral. By computing the difference between the
quadrature and the true integral, we can compute an optimal error
bound.

Let us now state our results precisely. Let $u(x,y,t)$ solve:
\begin{equation}
  \label{eq:wave}
  \partial_t^2 u(x,y,t) = \partial_x^2 u(x,y,t)+\Delta_y u(x,y,t)
\end{equation}
where $x \in \RR$ and $y \in \RR^{N-1}$ ($x$ is the normal direction,
$y$ the tangential directions).

We wish to solve \eqref{eq:wave} on $\RR^{N+1}$. The boundary will be
taken to be the surface $x=0$, and thus the approximation region will
be the region $\{ (x,y,z,t) : x \geq 0 \}$. We let $u_b(x,y,t)$ be the
approximation, solving \eqref{eq:wave} on the half-space. The boundary
conditions imposed are Hagstrom's:
\begin{equation}
  \label{eq:hagstromBoundary}
  \prod_{j=1}^n \left(\cos\left(\frac{j\pi}{n+1}\right) \partial_t - \partial_x \right)u_b(x,y,t)=0
\end{equation}

The main theorem is the following:
\begin{thm}
  \label{thm:main}
  Let $u(x,y,t)$ solve \eqref{eq:wave} on $\RR^{N+1}$, and
  $u_b(x,y,t)$ solve \eqref{eq:wave} with boundary conditions
  \eqref{eq:hagstromBoundary}. Then we have the following error bound:
  \begin{multline}
    \label{eq:errorBound}
    \abs{u(x,y,t) - u_b(x,y,t)}\\
    \leq \frac{\kmax}{3}\frac{\pi^4}{(n+1)^3}M(x)\left(2nt^2+9nt+n\kmax+8n+3\right) \\
    = O\left(\frac{\kmax}{n^{2}} (\kmax+t^{2}) \right)
  \end{multline}
\end{thm}

\section{Proof}

\subsection{The Exact Boundary}

We begin by reviewing the exact boundary conditions described in
\cite{MR1819643}. Applying the Laplace transform of \eqref{eq:wave} with respect to time
(letting $s$ be dual to $t$) and the Fourier transform with respect to
$y$ (with $\K$ dual to $y$) yields:
\begin{equation}
  \label{eq:FLwave}
  s^2 \U = \partial_x^2 \U  - \K^2 \U
\end{equation}
The solution to \eqref{eq:FLwave} is:
\begin{equation}
  \U(x)=A(s,\K)e^{\sqrt{s^2+|\K|}x}+B(s,\K)e^{-\sqrt{s^2+|\K|^2}x}
\end{equation}
The solutions with nonzero $A(s,\K)$ are nonphysical, since they
correspond to a wave coming from infinity to the object. Thus our
boundary conditions must imply $A(s,\K)= 0$. Such a boundary condition
is (in the frequency domain):
\begin{equation}
  \label{eq:exactD2N}
  \partial_x \U(x,\K,s)+\sqrt{s^2+|\K|^2}\U(x,\K,s)=0
\end{equation}
Of course, the operator $\sqrt{s^2+|\K|^2}$ is non-local in time and
space, so we will approximate it.

To reduce the dependence to a single variable, we make the
substitution $z=s/ \abs{\K}$, yielding:
\begin{equation*}
  \partial_{x} \U(x,\K,s)+ |\K|\sqrt{1+z^2} \U(x,\K,s) = 0
\end{equation*}

This boundary condition can be rewritten as:

\begin{equation}
  \partial_{x} \U(x,\K,s) + \abs{\K}\left(z+\frac{1}{z+\sqrt{1+z^2}} \right)\U(x,\K,s)=0
\end{equation}

Let $h(z) \equiv \abs{\K}/(z+\sqrt{1+z^2})$. We will invert the
Laplace transform, and shift the contour to surroung the singularities
of $h(z)$. Ths following lemma summarizes the necessary analyticity
properties of $h(z)$:

\begin{lemma}
  The function $h(z)$ is analytic on $\mathbb{C} \setminus [-i,i]$. In
  addition, the difference across the branch cut is given by:
  \begin{equation}
    \label{eq:1}
    \lim_{\epsilon\to 0}\left(h(z+\epsilon)-h(z-\epsilon)\right)=2|\K|\sqrt{1+z^2}
  \end{equation}
\end{lemma}

\begin{proof}
  The function $h(z)$ is well defined and analytic for $\Re z > 0$. It
  is strictly imaginary on $\{z : \Re z = 0 \text{and} \abs{z} > 1\}$.
  By the Schwartz reflection principle, it can be analytically
  continued to the left half plane, with a discontinuity along the
  line $[-i,i]$.

  An explicit calculation shows \eqref{eq:1}.
\end{proof}

We now reconstruct $u(x,y,t)$. This is done by inverting the Laplace
transform:

\begin{subequations}
  \begin{equation}
    \label{eq:2}
    u(x,y,t) = \frac{1}{(2\pi)^{(N+1)/2}} \int_{a+i \RR} e^{s t} \int_{\RR^{N-1}} e^{i y \cdot \K}  \U(x,\K, s) d\K ds
  \end{equation}

  \begin{equation}
    \label{eq:3}
    \Uf(x,\K,t) = \frac{1}{2\pi} \int_{a + i \RR} e^{s t} \U(x,\K,s) ds
  \end{equation}
\end{subequations}

And so, the integral we must approximate is
\begin{equation*}
  \int_{-i}^i 2|\K|\sqrt{1+z^2}f(z)e^{zt}dz.
\end{equation*}

\subsection{The Approximation}

We review the approximation itself, and how
\eqref{eq:hagstromBoundary} was derived. Our description follows
\cite{MR1819643} quite closely. We approximate
$\abs{\K}\sqrt{1+z^{2}}$ by:
\begin{equation}
  \abs{\K} \sqrt{1+z^2} = \abs{\K}\left(z+\frac{1}{z+\sqrt{1+z^2}} \right) \approx  \abs{\K}\left(z+\frac{1}{2z+\frac{1}{\ddots 2z}} \right)
\end{equation}
where the right hand side is the $n$'th iteration of the continued
fraction.

A straightforward computation shows that in the frequency domain,
\begin{equation*}
  \partial_x \U(x,\K,s) + \abs{\K}\left(z+\frac{1}{2z+\frac{1}{\ddots 2z}} \right)\U(x,\K,s)=0
\end{equation*}
corresponds to the boundary condition \eqref{eq:hagstromBoundary}. We
simplify this:

\begin{lemma}
  Let $\theta_{j}=j\pi/(n+1)$. Then we have the following formula:
  \begin{equation}
    \label{eq:continuedFractionApproximation}
    \frac{1}{2z+\frac{1}{\ddots 2z}}=\sum_{j=1}^n \frac{\sin^2\theta_j}{(n+1)(z-i\cos\theta_j)}
  \end{equation}
\end{lemma}

\begin{proof}
  Let $U_n(x)$ be the $n^{th}$ Chebyshev Polynomial of the $2^{nd}$
  kind and $P_n(z)$ be the successive numerators of the sequence of
  finite continued fractions ($P_0(z)=1, P_1(z)=2z, \ldots$).

  If we take $U_n(iz)$ then for $n$ odd we get $iP_n(z)$ and for $n$
  even we get $P_n(z)$, and the sequence of finite continued fractions
  is $(P_{n-1}(z))/(P_n(z))$ for $n\geq 1$.

  We consider the case case where $n$ is even; in this case, the
  finite continued fraction is $(iP_{n-1}(z))(P_n(z))$.  Thus, ratios
  of Chebychev polynomials of the $2^{nd}$ kind only differ from the
  finite continued fraction by multiplication by $i$, and will
  therefore have the same zeros.  \ The continued fraction will have
  poles where $P_n(z)$ is zero, for $n$ even.  That is, when
  $U_n(iz)=0$.

  We will take $z=i\cos\theta$.  Thus, we are looking for zeroes of
  $U_n(-\cos\theta)$ where $U_{n-1}(-\cos\theta)\neq 0$.

\begin{equation*}
  U_n(-\cos\theta)=U_n(\cos\theta)=U_n(x)=\frac{\sin(n+1)\theta}{\sin\theta},\quad \theta\neq 0,\pi,2\pi,\ldots
\end{equation*}
So $\sin(n+1)\theta=0$ are our solutions, and that is
$\theta=\frac{j\pi}{n+1}=\theta_j$ as claimed.

Thus, $U(i \cos \theta_{j})=0$ and hence $z=i \cos \theta_{j}$ are the
only poles of the continued fraction approximation. A direct
computation shows that the residues at the pole $i \cos \theta_{j}$ is
$(sin^2\theta_j)/(n+1)$.
\end{proof}

As the continued fraction is a close approximation to $\sqrt{1+z^2}$, we can use it to approximate an integral involving $\sqrt{1+z^2}$ by substituting the approximation, which is a rational function.  And so, in evaluating the integral around the branch cut, a finite sum which approximates this integral is given by the sum of the residues at the poles of the rational function above.

\subsection{The Error Bound}

First, we make the definition $|k|g(s/|k|)=\U(s,k)$.

\begin{prop}
  \label{prop:maintheorem}
  The following error bound holds.
  \begin{multline}
    \left|2|\K|\int_{-i}^i \sqrt{1+z^2} e^{zt}g(z)dz-2\pi i|\K|\sum_{j=1}^n g(z_j)e^{z_jt}\alpha_j\right|\\
    \leq\frac{\kmax}{3}\frac{\pi^4}{(n+1)^3}M(x)(2nt^2+9nt+n\kmax+8n+3)
  \end{multline}
  Here, $z_j=i \cos \theta_{j}$ are the positions of the poles,
  $\alpha_j$ the residue at $\theta_j$, $\kmax$ the maximal frequency
  under consideration, $M(x)$ is a pointwise upper bound on
  $\U,\partial_s\U$ and $\partial_s^2\U$ and $n$ is the order of the
  continued fraction approximation.
\end{prop}

We will need the following lemma

\begin{lemma}
  \begin{equation*}
    \left|\int_0^{\Delta x} f(x)dx-f(0)\Delta x\right|\leq \Delta x^2  f'(\xi),\quad \xi\in [0,\Delta x]
  \end{equation*}
\end{lemma}

\begin{proof}
  This follows immediately from the Intermediate Value Theorem.
\end{proof}

We will prove Lemma \ref{prop:maintheorem} above in several
intermediate steps.

\begin{prop}
  \begin{multline}
    \label{eq:big}
    \left|2k\int_{-i}^i\sqrt{1+z^2}e^{zt}g(z)dz-2\pi ik\sum_{j=1}^n g(z_j)e^{z_j t}\alpha_j\right|\\
    \leq\frac{2}{3}|\K|\Delta\theta^3\Bigg(3\left|\max_{s\in[i|\K|,i|\K|\cos\Delta\theta]}g\left(\frac{s}{|\K|},|\K|\right)\right|+3\left|\max_{s\in[-i|\K|\cos\Delta\theta,-i|\K|]}g\left(\frac{s}{|\K|},|\K|\right)\right|\\
    +\sum_{j=1}^n\Bigg|\max_{s\in[i|\K|\cos(\theta_j-\Delta\theta),i|\K|\cos(\theta_j+\Delta\theta)]}e^{\frac{s}{|\K|}t}\Bigg(g''\left(\frac{s}{|\K|},|\K|\right)\left(\frac{s}{|\K|}\right)^4\\
    +2g''\left(\frac{s}{|\K|},|\K|\right)\left(\frac{s}{|\K|}\right)^2+g''\left(\frac{s}{|\K|},|\K|\right)+2tg'\left(\frac{s}{|\K|},|\K|\right)\left(\frac{s}{|\K|}\right)^4\\
    +4tg'\left(\frac{s}{|\K|},|\K|\right)\left(\frac{s}{|\K|}\right)^2+2tg'\left(\frac{s}{|\K|},|\K|\right)+t^2g\left(\frac{s}{|\K|},|\K|\right)\left(\frac{s}{|\K|}\right)^4\\
    +2t^2g\left(\frac{s}{|\K|},|\K|\right)\left(\frac{s}{|\K|}\right)^2+t^2g\left(\frac{s}{|\K|},|\K|\right)+5\left(\frac{s}{|\K|}\right)^3g'\left(\frac{s}{|\K|},|\K|\right)\\
    +5g'\left(\frac{s}{|\K|},|\K|\right)\frac{s}{|\K|}+5\left(\frac{s}{|\K|}\right)^3tg\left(\frac{s}{|\K|},|\K|\right)+5tg\left(\frac{s}{|\K|},|\K|\right)\frac{s}{|\K|}\\
    +4\left(\frac{s}{|\K|}\right)^2g\left(\frac{s}{|\K|},|\K|\right)+2g\left(\frac{s}{|\K|},|\K|\right)\Bigg)\Bigg|\Bigg)
  \end{multline}
  In this equation $z_j$ are the positions of the poles, $\alpha_j$ are the
  residues at the poles, $\Delta \theta=\frac{\pi}{2(n+1)}$, and
  $\alpha(\theta)=g(i\cos\theta)e^{it\cos\theta}\sin^2\theta$.
\end{prop}

\begin{proof}
  We first change variables to $z=i\cos\theta$. Using the fact that
  \begin{equation*}
    \sqrt{1+z^2}=\frac{1}{2z+\frac{1}{2z+\ddots}}
  \end{equation*}
  we can approximate $\sqrt{1+z^{2}}$ by taking a truncated continued
  fraction. This yields:
  \begin{equation}
    \label{eq:15}
    |k|\left|2\int_0^\pi g(i\cos\theta)e^{it\cos\theta}\sin^2\theta
      d\theta-2\pi i\sum_{j=1}^n
      g(i\cos\theta_j)e^{it\cos\theta_j}\frac{\sin^2\theta_j}{n+1}\right|
  \end{equation}

  We define $f(z)=g(z)e^{zt}$ and $\Delta\theta=\frac{\pi}{2(n+1)}$
  and expand the integral around each pole to obtain
  \begin{multline}
    \label{eq:16}
    \eqref{eq:15} =
    |k|\Bigg|2i\int_0^{\Delta\theta}f(i\cos\theta)\sin^2\theta d\theta+2i\int_{\pi-\Delta\theta}^\pi f(i\cos\theta)\sin^2\theta d\theta \\
    +2i\sum_{j=1}^n\Bigg(\int_{\theta_j-\Delta\theta}^{\theta_j+\Delta\theta}f(i\cos\theta)\sin^2\theta
    d\theta-\pi g(i\cos\theta_j)\frac{\sin^2\theta_j}{n+1}\Bigg)\Bigg|
  \end{multline}
  To simplify further, we substitute
  $\alpha(\theta)=f(i\cos\theta)\sin^2\theta$ cancel terms, and use
  the fact that
  $\int_{\theta_j-\Delta\theta}^{\theta_j+\Delta\theta}\alpha'(\theta_j)(\theta-\theta_j)d\theta=0$
  to get
  \begin{multline}
    \label{eq:calc1}
    \eqref{eq:16} =
    |k|\Bigg|2i\int_0^{\Delta\theta}\alpha(\theta) d\theta+2i\int_{\pi-\Delta\theta}^\pi \alpha(\theta) d\theta+2i\sum_{j=1}^n\Bigg(\int_{\theta_j-\Delta\theta}^{\theta_j+\Delta\theta}\big(\alpha(\theta)-\alpha(\theta_j)\\
    -\alpha'(\theta_j)(\theta-\theta_j)\big)\Bigg)d\theta\Bigg|
  \end{multline}
By the triangle inequality, and the mean value theorem, we have
\begin{multline}
  \label{eq:calc2}
  \eqref{eq:calc1}\leq2|k|\Bigg(\int_0^{\Delta\theta}\left|\alpha(\theta)\right|d\theta+\int_{\pi-\Delta\theta}^\pi \left|\alpha(\theta)\right|d\theta\\
  +\frac{1}{3}\sum_{j=1}^n\left|\max_{\xi\in[\theta_j-\Delta\theta,\theta_j+\Delta\theta]}\alpha''(\xi)\right|\Delta\theta^3\Bigg)
\end{multline}
To deal with the ends of the integral, we substitute $\alpha$ and $f$
back into the integrals near the endpoints. We then use the fact that $\int_a^b f\leq \max_{x\in[a,b]}f(x)(b-a)$ and $\abs{\sin \theta} \leq \abs{\theta}$ and $\abs{\sin (\pi-\theta)} \leq \abs{\pi-\theta}$ to obtain:
\begin{multline}
  \eqref{eq:calc2}
  \leq
  2|k|\Bigg(\left|\max_{[0,\Delta\theta]}g(i\cos\theta)\right|\Delta\theta^3+\left| \max_{[\pi-\Delta\theta,\pi]}g(i\cos\theta)\right|\Delta\theta^3\\
  +\frac{1}{3}\sum_{j=1}^n\left|\max_{\xi\in[\theta_j-\Delta\theta,\theta_j+\Delta\theta]}\alpha''(\xi)\right|\Delta\theta^3\Bigg)
\end{multline}
Upon substitution back and simplification, this becomes
\begin{multline}
  \label{eq:14}
  \eqref{eq:calc2} \leq
  \frac{2}{3}|k|\Delta\theta^3\Bigg(3\left|\max_{[i,i\cos\Delta\theta]}g(z)\right|+3\left|\max_{[-i\cos\Delta\theta,-i]}g(z)\right|\\
  +\sum_{j=1}^n\Bigg|\max_{[i\cos(\theta_j-\Delta\theta),i\cos(\theta_j+\Delta\theta)]}e^{zt}\Bigg(g''(z)z^4+2g''(z)z^2+g''(z)+2tg'(z)z^4\\
  +4tg'(z)z^2+2tg'(z)+t^2g(z)z^4+2t^2g(z)z^2+t^2g(z)+5z^3g'(z)+5g'(z)z+5z^3tg(z)\\
  +5tg(z)z+4z^2g(z)+2g(z)\Bigg)\Bigg|\Bigg)
\end{multline}
We also know that $g(z)=g\left(s/|\K|,k\right)$. All the derivatives in \eqref{eq:13} are in $s/|\K|$.  And so we get:
\begin{multline}
  \label{eq:13}
  \eqref{eq:14} \leq
  \frac{2}{3}|\K|\Delta\theta^3\Bigg(3\left|\max_{s\in[i|\K|,i|\K|\cos\Delta\theta]}g\left(\frac{s}{|\K|},|\K|\right)\right|+3\left|\max_{s\in[-i|\K|\cos\Delta\theta,-i|\K|]}g\left(\frac{s}{|\K|},|\K|\right)\right|\\
  +\sum_{j=1}^n\Bigg|\max_{s\in[i|\K|\cos(\theta_j-\Delta\theta),i|\K|\cos(\theta_j+\Delta\theta)]}e^{\frac{s}{|\K|}t}\Bigg(g''\left(\frac{s}{|\K|},|\K|\right)\left(\frac{s}{|\K|}\right)^4\\
  +2g''\left(\frac{s}{|\K|},|\K|\right)\left(\frac{s}{|\K|}\right)^2
  +g''\left(\frac{s}{|\K|},|\K|\right)+2tg'\left(\frac{s}{|\K|},|\K|\right)\left(\frac{s}{|\K|}\right)^4\\
  +4tg'\left(\frac{s}{|\K|},|\K|\right)\left(\frac{s}{|\K|}\right)^2+2tg'\left(\frac{s}{|\K|},|\K|\right)+t^2g\left(\frac{s}{|\K|},|\K|\right)\left(\frac{s}{|\K|}\right)^4\\
  +2t^2g\left(\frac{s}{|\K|},|\K|\right)\left(\frac{s}{|\K|}\right)^2+t^2g\left(\frac{s}{|\K|},|\K|\right)+5\left(\frac{s}{|\K|}\right)^3g'\left(\frac{s}{|\K|},|\K|\right)\\
  +5g'\left(\frac{s}{|\K|},|\K|\right)\frac{s}{|\K|}+5\left(\frac{s}{|\K|}\right)^3tg\left(\frac{s}{|\K|},|\K|\right)+5tg\left(\frac{s}{|\K|},|\K|\right)\frac{s}{|\K|}\\
  +4\left(\frac{s}{|\K|}\right)^2g\left(\frac{s}{|\K|},|\K|\right)+2g\left(\frac{s}{|\K|},|\K|\right)\Bigg)\Bigg|\Bigg)
\end{multline}
\end{proof}

Now we can prove the final bound, and complete the proof of the main
theorem. Once Proposition \ref{prop:maintheorem} is proven, this
implies the main result by \eqref{eq:2} and \eqref{eq:3}.

\begin{proofof}{Proposition \ref{prop:maintheorem}}

  $\partial_s \U(s,k)=\partial_s kg(s/k,k)=kD_1 \frac{1}{k} g(s/k,k)=D_1
  g(\frac{s}{|\K|},k)$.

  $\partial_s^2 \U(s,k)=\partial_s^2(kg(s/k,k))=k\partial_s^2 g(s/k,k)=1/k D_1^2
  g(\frac{s}{|\K|},k)$

  So $D_1 g=\partial_s \U$ and $D_1^2 g=|\K|\partial_s^2\U$

  Thus, \eqref{eq:big} above can be simplified to:

\begin{multline}
\label{eq:4}
  \frac{2}{3}\Delta\theta^3\Bigg(3\Bigg|\max_{s\in[i|\K|,i|\K|\cos\Delta\theta]}\U(s,|\K|)\Bigg|+3\Bigg|\max_{s\in[-i|\K|\cos\Delta\theta,-i|\K|]}\U(s,|\K|)\Bigg|\\
  +\sum_{j=1}^n\Bigg|\max_{s\in[i|\K|\cos(\theta_j-\Delta\theta),i|\K|\cos(\theta_j+\Delta\theta)]}e^{\frac{s}{|\K|}t}\Bigg(|\K|\partial_s^2\U(s,|\K|)\left(\frac{s}{|\K|}\right)^4\\
  +2|\K|\partial_s^2\U(s,|\K|)\left(\frac{s}{|\K|}\right)^2+|\K|\partial_s^2\U(s,|\K|)+2t\partial_s \U(s,|\K|)\left(\frac{s}{|\K|}\right)^4\\
  +4t\partial_s\U(s,|\K|)\left(\frac{s}{|\K|}\right)^2+2t\partial_s\U(s,|\K|)+t^2\U\left(s,|\K|\right)\left(\frac{s}{|\K|}\right)^4+2t^2\U(s,|\K|)\left(\frac{s}{|\K|}\right)^2\\
  +t^2\U(s,|\K|)+5\left(\frac{s}{|\K|}\right)^3\partial_s\U(s,|\K|)+5\partial_s\U(s,|\K|)\frac{s}{|\K|}+5\left(\frac{s}{|\K|}\right)^3t\U(s,|\K|)\\
  +5t\U(s,|\K|)\frac{s}{|\K|}+4\left(\frac{s}{|\K|}\right)^2\U(s,|\K|)+2\U(s,|\K|)\Bigg)\Bigg|\Bigg)
\end{multline}

If we find the maximum for each term independently, we will obtain an
upper bound for this.  Noting that $\cos\theta$ is monotonic
decreasing on $[0,\pi]$, we obtain:

\begin{multline}
  \label{eq:5}
  \eqref{eq:4} \leq
  \frac{2}{3}\Delta\theta^3\Bigg(3\left|\max_{s\in[i|\K|,i|\K|\cos\Delta\theta]}\U\left(s,|\K|\right)\right|+3\left|\max_{s\in[-i|\K|\cos\Delta\theta,-i|\K|]}\U\left(s,|\K|\right)\right|\\
  +\sum_{j=1}^n\Bigg|\max_{s\in[i|\K|\cos(\theta_j-\Delta\theta),i|\K|\cos(\theta_j+\Delta\theta)]}e^{i\cos(\theta_j-\Delta\theta)t}\bigg(|\K|\partial_s^2\U\left(s,|\K|\right)\cos^4(\theta_j-\Delta\theta)\\
  -2|\K|\partial_s^2\U\left(s,|\K|\right)\cos^2(\theta_j-\Delta\theta)+|\K|\partial_s^2\U\left(s,|\K|\right)+2t\partial_s \U\left(s,|\K|\right)\cos^4(\theta_j-\Delta\theta)\\
  -4t\partial_s\U\left(s,|\K|\right)\cos^2(\theta_j-\Delta\theta)+2t\partial_s\U\left(s,|\K|\right)+t^2\U\left(s,|\K|\right)\cos(\theta_j-\Delta\theta)\\
  -2t^2\U\left(s,|\K|\right)\cos^2(\theta_j-\Delta\theta)+t^2\U\left(s,|\K|\right)-5i\cos^3(\theta_j-\Delta\theta)\partial_s\U\left(s,|\K|\right)\\
  +5\partial_s\U\left(s,|\K|\right)i\cos(\theta_j-\Delta\theta)-5i\cos^3(\theta_j-\Delta\theta)t\U\left(s,|\K|\right)+5t\U\left(s,|\K|\right)i\cos(\theta_j-\Delta\theta)\\
  -4\cos^2(\theta_j-\Delta\theta)\U\left(s,|\K|\right)+2\U(s,|\K|)\bigg)\Bigg|\Bigg)
\end{multline}

We also know that $\U(s,x,k)$ is bounded, so let $U(x)$ be an upper
bound, let $U'(x)$ be an upper bound of $\partial_s\U$ and $U''(x)$ and upper
bound of $\partial_s^2 \U$.  Then let $M(x)$ be the maximum of these
functions.  Now, things simplify further to:

\begin{multline}
  \label{eq:6}
  \eqref{eq:5}  \leq
  \frac{2}{3}\Delta\theta^3M(x)\Bigg(6+\sum_{j=1}^n\Bigg|e^{i\cos(\theta_j-\Delta\theta)t}\bigg(|\K|\cos^4(\theta_j-\Delta\theta)-2|\K|\cos^2(\theta_j-\Delta\theta)+|\K|\\
  +2t\cos^4(\theta_j-\Delta\theta)-4t\cos^2(\theta_j-\Delta\theta)+2t+t^2\cos(\theta_j-\Delta\theta)-2t^2\cos^2(\theta_j-\Delta\theta)+t^2\\
  -5i\cos^3(\theta_j-\Delta\theta)+5i\cos(\theta_j-\Delta\theta)-5i\cos^3(\theta_j-\Delta\theta)t+5ti\cos(\theta_j-\Delta\theta)\\
  -4\cos^2(\theta_j-\Delta\theta)+2\bigg)\Bigg|\Bigg)
\end{multline}
Applying the triangle inequality, we obtain the following as a bound.
\begin{multline}
  \label{eq:7}
  \eqref{eq:6} \leq
  \frac{2}{3}\Delta\theta^3M(x)\Bigg(6+\sum_{j=1}^n\big(|\K|\cos^4(\theta_j-\Delta\theta)-2|\K|\cos^2(\theta_j-\Delta\theta)+|\K|\\
  +2t\cos^4(\theta_j-\Delta\theta)-4t\cos^2(\theta_j-\Delta\theta)+2t+t^2|\cos(\theta_j-\Delta\theta)|-2t^2\cos^2(\theta_j-\Delta\theta)\\
  +t^2-5|\cos^3(\theta_j-\Delta\theta)|+5|\cos(\theta_j-\Delta\theta)|-5|\cos^3(\theta_j-\Delta\theta)|t\\
  +5t|\cos(\theta_j-\Delta\theta)|-4\cos^2(\theta_j-\Delta\theta)+2\big)\Bigg)
\end{multline}
Introducing $\beta_j=|\cos(\theta_j-\Delta\theta)|$, this can be
written as
\begin{multline}
  \label{eq:8}
  \eqref{eq:7} =
  \frac{2}{3}\Delta\theta^3M(x)\Bigg(6+\sum_{j=1}^n\big(|\K|\beta_j^4-2|\K|\beta_j^2+|\K|+2t\beta^4_j-4t\beta^2_j+2t+t^2\beta_j\\
  -2t^2\beta^2_j +t^2-5\beta_j^3
  +5\beta_j-5\beta^3_jt+5t\beta_j-4\beta^2_j+2\big)\Bigg)
\end{multline}
Now, we integrate in $\K$ over the circle of radius $\kmax$.  This
translates to integrating $|\K|$ from $0$ to $\kmax$ and multiplying
by $2\pi$..  This gives us
\begin{multline}
  \label{eq:9}
  \eqref{eq:8} \leq
  \frac{4\pi}{3}\Delta\theta^3M(x)\Bigg(6\kmax +\frac{\kmax}{2}\sum_{j=1}^n\Big(\kmax\beta_j^4-2\kmax\beta_j^2\\
  +\kmax+4t\beta^4_j-8t\beta^2_j+4t+2t^2\beta_j-4t^2\beta^2_j\\
  +2t^2-10\beta_j^3+10\beta_j-10\beta^3_jt+10t\beta_j-8\beta^2_j+4\Big)\Bigg)
\end{multline}
And this becomes
\begin{multline}
  \label{eq:10}
  \eqref{eq:9} \leq
  \frac{2\kmax\pi}{3}\Delta\theta^3M(x)\Bigg[12+4n+2nt^2+4nt+\kmax n \\
  +\sum_{j=1}^n\Big((\kmax+4t)\beta_j^4-10(t+1)\beta_j^3\\
  -2(2t^2+4t+\kmax+4)\beta_j^2+2(t^2+5t+5)\beta_j\Big)\Bigg]
\end{multline}
Breaking up the sum yields:
\begin{multline}
  \label{eq:11}
  \eqref{eq:10} =
  \frac{2\kmax\pi}{3}\Delta\theta^3M(x)(12+4n+2nt^2+4nt+\kmax
  n\\+(\kmax+4t)\sum_{j=1}^n\beta_j^4
  -10(t+1)\sum_{j=1}^n\beta_j^3\\-2(2t^2+4t+\kmax+4)
  \sum_{j=1}^n\beta_j^2+2(t^2+5t+5)\sum_{j=1}^n\beta_j)
\end{multline}
Now, we notice that $\beta_j=|\cos\phi|$ for some $\phi$ and that
$|\cos\phi|\leq 1$.  This finally allows us to remove the $j$
dependence of the terms inside the sum, and we obtain, after
substituting $\Delta\theta$ back in:
\begin{equation}
\label{eq:12}
\eqref{eq:11} \leq
  \frac{\kmax}{3}\frac{\pi^4}{(n+1)^3}M(x)\left(2nt^2+9nt+n\kmax+8n+3\right)
\end{equation}
\end{proofof}

\subsection{Improving the quadrature}

The result we describe here depends on the following idea: $\sqrt{1+z^{2}}$ has a branch cut on the region $[-i,i]$. The rational function approximation can be expanded as a sum of first order poles, as per \eqref{eq:continuedFractionApproximation}. Integrating an analytic function against this sum of poles (around a contour encircling $[-i,i]$ yields a sum of the form $\sum_{n} w_{n} f(z_{n})$, which approximates the integral of $f(z) \sqrt{1+z^{2}}$ around a contour encircling $[-i,i]$. In particular, this is a second order quadrature.

A natural line of inquiry is to ask is whether higher order quadratures can be used, simply by discarding the rational function approximation, and merely choosing a sum of poles according to some appropriate quadrature rule. We conjecture that this can be done.

{\bf Acknowledgements: } A. Soffer and C. Stucchio were supported by NSF grant DMS01-00490. C.M. Siegel was supported by the 2005 DIMACS/Rutgers Research Experience for Undergraduates. We also acknowledge that some of this work may duplicate recent results of Tom Hagstrom, Bradley Alpert and Leslie Greengard.
\bibliographystyle{plain}
\bibliography{reu}
\end{document}